\documentclass[a4paper, 10pt]{amsart}

\usepackage{amsmath, amscd, amsfonts, amssymb, amsthm, latexsym, url, color, todonotes, bm, framed, rotating} 
\usepackage{graphicx}
\usepackage[left=1.3in, right=1.2in, top=1.3in, bottom=1.3in, includefoot, headheight=13.6pt]{geometry}
\input{xy}
\xyoption{all}

\newtheorem{theorem}{Theorem}[section]
\newtheorem{lemma}[theorem]{Lemma}
\newtheorem{corollary}[theorem]{Corollary}

\newtheorem{definition}[theorem]{Definition}

\theoremstyle{definition}
\newtheorem{remark}[theorem]{Remark}
\newtheorem{example}[theorem]{Example}

\newcommand{\xysquare}[8]{
\[\xymatrix{
#1 \ar@{#5}[r] \ar@{#6}[d] & #2 \ar@{#7}[d]\\
#3 \ar@{#8}[r] & #4
}\]
}

\DeclareMathOperator*{\holim}{\operatorname*{holim}}

\newcommand{\bb}{\mathbb}

\newcommand{\comment}[1]{}

\newcommand{\ep}{\varepsilon}

\newcommand{\into}{\hookrightarrow}
\newcommand{\isoto}{\stackrel{\simeq}{\to}}
\newcommand{\Isoto}{\stackrel{\simeq}{\longrightarrow}}

\newcommand{\onto}{\twoheadrightarrow}
\newcommand{\op}{\operatorname}

\renewcommand{\phi}{\varphi}
\newcommand{\quis}{\stackrel{\sim}{\to}}

\newcommand{\roi}{\mathcal{O}}

\newcommand{\sub}[1]{{\mbox{\rm \scriptsize #1}}}

\newcommand{\To}{\longrightarrow}

\newcommand{\xto}{\xrightarrow}

\newcommand{\TC}{T\!C}

\newcommand{\cal}{\mathcal}
\renewcommand{\hat}{\widehat}
\renewcommand{\frak}{\mathfrak}

\renewcommand{\ker}{\operatorname{Ker}}
\newcommand{\coker}{\operatorname{Coker}}

\DeclareMathOperator{\Spec}{Spec}

\DeclareMathOperator{\Tor}{Tor}

\newcommand{\xTo}[1]{\stackrel{#1}{\To}}

\begin{document}

\title[A historical overview of pro cdh descent in algebraic $K$-theory]{A historical overview of pro cdh descent in algebraic $K$-theory and its relation to rigid analytic varieties}

\author{M.~Morrow}

\begin{abstract}
This note was prepared after the workshop on cdh descent and algebraic $K$-theory in Hara-mura, Japan (1 -- 5 Sept.~2016), to complement the author's talk on the history and applications of pro cdh descent. In the final section pro cdh descent is used to define $K$-groups of rigid analytic varieties.

This note was revised and released in light of Kerz, Strunk, and Tamme establishing pro cdh descent in general \cite{KerzStrunkTamme2016}.
\end{abstract}

\date{}

\maketitle


\setcounter{section}{-1}

\section{Introduction}
We begin by recalling that an {\em abstract blow-up square} $\Sigma$ of schemes (always quasi-compact and quasi-separated in this note)
\[\xymatrix{
Y'\ar[d]\ar[r] & X'\ar[d]^f\\
Y\ar[r] & X
}\]
is a cartesian square in which the horizontal arrows are closed embeddings and $f$ is a proper morphism which induces an isomorphism $X'\setminus Y'\isoto X\setminus Y$. For example, $X'$ could be the blow-up of $X$ with centre $Y$. We say that $K$-theory (always non-connective in this note) satisfies {\em pro cdh descent} with respect to this square if and only if the resulting map of pro abelian relative $K$-groups \[\{K_n(X,Y_s)\}_s\To \{K_n(X',Y_s')\}_s\] is an isomorphism for all $n\in\bb Z$.\footnote{Equivalently, the pro abelian group $\{K_n(X,X',Y_s)\}_s$ vanishes for each $n\in\bb Z$, where $K_n(X,X',Y_s):=\pi_n\op{hofib}(K(X,Y_s)\to K(X',Y'_s))$. The definition is also equivalent to asking that the square of pro spectra
\[\xymatrix{
K(X)\ar[d]\ar[r] & \{K(Y'_s)\}_s\ar[d]\\
K(X')\ar[r] & \{K(Y_s)\}_s
}\]
be homotopy cartesian with respect to Fausk--Isaksen's $\pi_*$-model structure on the category of pro spectra \cite[\S1.1]{FauskIsaksen2007}. Indeed, weak equivalences of pro spectra with respect to this model structure are maps which induce isomorphism on all pro homotopy groups, together with a stability condition on the behaviour of $s$ (the index of the pro spectra) as $n$ (the degree of the homotopy group) tends to $-\infty$ (to be more precise, this stability condition is that the map should be isomorphic to a level-wise $N$-weak-equivalence for some $N\in\bb Z$). But this stability condition is redundant in our situation since $K_n(Y_s)=K_n(Y)$ and $K_n(Y_s')=K_n(Y')$ for $n\ll0$ (independently of $s$); indeed, since our schemes are qcqs, this follows from the nil-invariance of negative $K$-theory on rings and Zariski descent.
} Here $Y_s$ denote the $s-1^\sub{st}$-infinitesimal thickening of $Y$ inside $X$, and similarly for $Y'$.

It was an open problem until very recently to prove that $K$-theory satisfies pro cdh descent with respect to any abstract blow-up square of Noetherian schemes, but this has now been positively settled by Kerz, Strunk, and Tamme \cite{KerzStrunkTamme2016}:

\begin{theorem}[Kerz--Strunk--Tamme]
Suppose $\Sigma$ is an abtract blow-up square of Noetherian schemes; then $K$-theory satisfies pro cdh descent with respect to $\Sigma$.
\end{theorem}

In this note we will review some history concerning pro cdh descent (and pro excision) and explain how it relates to the $K$-theory of rigid analytic varieties.

\subsection*{Acknowledgements}
I am extremely grateful to the organisers Wataru Kai, Hiroyasu Miyazaki, and Ryomei Iwasa for their invitation to aforementioned workshop in Hara-mura, and thank all the participants for their questions and interest in the subject.

I would also like to thank Moritz Kerz, Shuji Saito, and Georg Tamme for various discussions about their own ongoing approach to the $K$-theory of rigid analytic spaces.

\section{Classical results in low degree}
We begin with a review of some more ``classical'' (1960s--2000s) results in the subject concerning pro excision and pro cdh descent in low degree, many of which have elementary proofs.

\subsection{Pro excision}
Suppose that $f:A\to B$ is a homomorphism of (always commutative, in this note, for simplicity) rings and $I\subseteq A$ is an ideal with the following properties: the homomorphism $f$ is injective on $I$ (i.e., $I\cap\ker f=0$), and the image $f(I)$ (which we henceforth identify with $I$) is an ideal of $B$. For brevity, it is convenient to call this an {\em excision situation}. Note that if $A\to B$, $I$ is an excision situation, then so is $A\to B$, $I^s$ for each $s\ge1$.

\begin{example}
The following are the two main examples:
\begin{enumerate}
\item $A\subseteq B$ is an extension of rings and $I$ is any ideal of $B$ which is contained inside $A$, such as the {\em conductor ideal} $I=\op{Ann}_AB/A$ ($=\{a\in A:aB\subseteq A\}$).
\item $A\onto B=A/J$ is a surjection of rings and $I$ is any ideal of $A$ satisfying $I\cap J=0$.
\end{enumerate}
In practice it is typically sufficient to treat these two special cases, since if $f:A\to B$, $I$ is an excision situation, then so are $A\onto f(A)$, $I$ and $f(A)\into B$, $f(I)$.
\end{example}

The resulting square
\[\xymatrix{
A\ar[r]^f\ar[d] & B \ar[d] \\
A/I\ar[r] & B/I
}\]
is often called a Milnor, or excision, square\footnote{The reader may wish to check that this square is both cartesian and cocartesian in the category of rings (as well as in the categories of $A$-modules and groups.} in honour of \cite[\S2]{Milnor1971}. We say that $K$-theory satisfies {\em pro excision} with respect to the given situation if and only if the resulting map of pro abelian groups $\{K_n(A,I^s)\}_s\to \{K_n(B,I^s)\}_s$ is an isomorphism for all $n\in\bb Z$.\footnote{Equivalently, the pro abelian group $\{K_n(A,B,I^s)\}_s$ vanishes for each $n\in\bb Z$, where $K_n(A,B,I^s):=\pi_n\op{hofib}(K(A,I^s)\to K(B,I^s))$.}

\begin{remark}
Note that if $f$ is a finite morphism then applying $\Spec$ to the above Milnor square yields an abstract blow-up square, and the associated pro cdh descent and pro excision problems are equivalent. Thus the following results on pro excision may be seen as the earliest manifestations of pro cdh descent.
\end{remark}

The following theorem and corollary summarise the main classical results concerning (pro) excision:

\begin{theorem}\label{theorem_excision}
Let $f:A\to B$, $I$ be an excision situation. Then:
\begin{enumerate}
\item $K_n(A,I)\to K_n(B,I)$ is an isomorphism for all $n\le0$ (Bass, 1968).
\item $K_1(A,I)\to K_1(B,I)$ is surjective and its kernel is a quotient of the abelian group $B/f(A)\otimes_{\bb Z}I/I^2$ (Milnor, Swan, 1971).
\item The birelative $K$-group $K_1(A,B,I)$ is isomorphic to $\Omega_{B/A}^1\otimes_{B}I/I^2$ (Geller--Weibel 1983).
\end{enumerate}
\end{theorem}
\begin{proof}[Notes on the proofs]
The case $n=0$ of (i) may be found as an exercise in Weibel's $K$-book \cite[Exs.~II.2.3 \& II.2.4]{Weibel2013}; the case $n<0$ then follows by the usual descending induction, noting that $A[T]\to B[T]$, $I A[T]$ is also an excision situation.

If $f$ is surjective, then it was proved by Milnor \cite[Lem.~6.3]{Milnor1971} that $K_1(A,I)\isoto K_1(B,I)$, and by Swan \cite[Thm.~7.1]{Swan1971} that $K_2(A,I)\onto K_2(B,I)$. Therefore $K_1(A,B,I)\isoto K_1(f(A),B,I)$, which reduces (ii) and (iii) to the case when $f$ is injective; we therefore henceforth assume that $I\subseteq A\subseteq B$.

(ii): Swan \cite[\S4]{Swan1971} used explicit calculations with Steinberg symbols to show that the sequence \[B/A\otimes_{\bb Z}I/I^2\xTo{\psi} K_1(A,I)\To K_1(B,I)\To 0\] is exact. Here the map $\psi$ is defined by sending an element $b\otimes x$ to the class (known as a Mennicke symbol) in $K_1(A,I)$ associated to the matrix 
\[\begin{pmatrix}1-bx & x \\ -b^2x & 1+bx\end{pmatrix}\in \op{SL}_2(A,I).\]

(iii): Vorst \cite[Thm.~2.5]{Vorst1979}, by observing some additional relations among Mennicke symbols, improved Swan's result by showing that $\psi$ actually factors through the surjection \[B/A\otimes_{\bb Z}I/I^2\To \Omega_{B/A}^1\otimes_BI/I^2,\qquad b\otimes x\mapsto db\otimes x,\] thereby defining the ``Swan--Vorst'' map $\ep:\Omega_{B/A}^1\otimes_BI/I^2\to K_1(A,I)$. Finally, it was shown by Geller--Weibel \cite[Thm.~1.1]{GellerWeibel1983} (more calculations with Mennicke symbols) that the Swan--Vorst map could be identified with the canonical map $K_1(A,B,I)\to K_1(A,I)$.
\end{proof}

\begin{corollary}\label{corollary_pro_excision<2}
$K$-theory satisfies pro excision in degrees $\le 1$; i.e., if $f:A\to B$, $I$ is an excision situation, then the pro abelian group $\{K_n(A,B,I^s)\}$ vanishes for $n\le1$.
\end{corollary}
\begin{proof}
The birelative group $K_n(A,B,I^s)$ vanishes for each fixed $s\ge1$ and $n\le0$ by (i) and the surjectivity part of (ii) of the previous theorem (recall that $A\to B$, $I^s$ is an excision situation for each $s\ge1$). Concerning the case $n=1$, the desired vanishing follows from (iii) and the observation that the pro abelian group $\{\Omega_{B/A}^1\otimes_BI^s/I^{2s}\}_s$ is zero (since each map $I^{2s}/I^{4s}\to I^s/I^{2s}$ is obviously zero).
\end{proof}

\begin{remark}[Failure of excision]
It is not hard to use Theorem \ref{theorem_excision}(iii) to find an example of an excision situation, with $f$ injective, for which $K_1(A,B,I)\neq 0$. For example, already in 1971, Swan \cite[\S3]{Swan1971} showed that in the situation $B:=\bb Z[\zeta_p]\supset A:=B+pB\supset I:=pB$, the surjection $K_1(A,I)\to K_1(B,I)$ is not an isomorphism. Therefore the passage to the pro world by taking all powers of $I$ seems to be essential to understanding excision and cdh descent phenomena in algebraic $K$-theory.
\end{remark}

\begin{remark}[Validity of excision]\label{remark_SW}
In contrast to the previous remark, the following theorem was proved by Suslin \cite{Suslin1995} in 1995, building on early work joint with Wodzicki \cite{Suslin1992}: if $f:A\to B$, $I$ is an excision situation, $N>0$ is fixed, and $\Tor_*^{\bb Z\ltimes I}(\bb Z,\bb Z)=0$ for all $0<n\le N$, then $K_*(A,B,I)=0$ for all $n\le N$. Unfortunately, the Tor vanishing hypothesis is rarely satisfied for the types of excision situations which appear in algebraic geometry; this result is more appropriate for Banach algebras, infinite-dimensional function algebras, etc.
\end{remark}

\subsection{Pro cdh descent}
The goal of this section is to prove Theorem \ref{theorem_Weibel}, which is historically the first interesting case of pro cdh descent which does not follow trivially from the above classical results on pro excision. We must first prove the following lemma, which follows in a relatively straightforward way from Corollary \ref{corollary_pro_excision<2}:

\begin{lemma}\label{lemma_pro_cdh_in_finite_case1}
Let $\Sigma$ be an abstract blow-up square of Noetherian schemes in which $f$ is a finite morphism, and assume that $X$ admits a cover by $1+d$ affine open subschemes. Then $K$-theory satisfies pro cdh descent with respect to $\Sigma$ in degree $\le 1-d$, i.e., the pro abelian group $\{K_n(X,X',Y_s)\}_s$ vanishes for $n\le1-d$.
\end{lemma}
\begin{proof}
By induction on $d$, using the Mayer--Vietoris property for the presheaf of spectra  $X\supseteq U\mapsto K(U,X'\times_XU,Y_s\times_XU)$, we may assume that $X$ is affine. Then $X'$ is also affine and $f$ corresponds to a finite morphism of rings $f:A\to B$. Since $\Sigma$ is an abstract blow-up square, the induced map on localisations $A_\frak p\to B\otimes_AA_\frak p$ is an isomorphism for all prime ideals $\frak p\subseteq A$ which do not contain $I:=I(Y)$; hence the finitely generated $A$-modules $\ker f$ and $\coker f$ are killed by $I^s$ for $s\gg0$. It then follows from the Artin--Rees lemma (resp.~a trivial argument) that $I^s\cap\ker f=0$ (resp.~$f(I^s)$ is an ideal of $B$) for $s\gg0$. In conclusion, possibly after replacing $Y$ by an infinitesimal thickening, the abstract blow-up square $\Sigma$ is $\Spec$ of a Milnor square and hence the desired pro cdh descent reduces to Corollary \ref{corollary_pro_excision<2}.
\end{proof}

\begin{remark}
If $\Sigma$ is an abstract blow-up square of Noetherian schemes in which $f$ is a finite morphism, then $K$-theory satisfies pro cdh descent with respect to $\Sigma$ in degrees $1-\dim Y$ (which may be better or worse than the bound offered by the previous lemma). To prove this, argue as in the previous lemma but replace the initial induction by the observation that each descent spectral sequence \[E_2^{pq}(s)=H^p(X,\cal K_{-q,(X,X',Y_s))})\Longrightarrow K_{-p-q}(X,X',Y_s)\] is supported in the range $p\le\dim Y$, since the sheaf $\cal K_{-q,(X,X',Y_s)}$ is supported on $Y$.
\end{remark}

The following result is implicitly due to Weibel \cite[\S3--4]{Weibel2001}, though he did not explicitly state it, and was used by Krishna--Srinivas \cite{Krishna2002} to study zero-cycles on singular surfaces:

\begin{theorem}\label{theorem_Weibel}
Let $\Sigma$ be an abstract blow-up square in which $X$ is an excellent, normal, $2$-dimensional scheme, $X'$ is regular, and $Y$ is $0$-dimensional. Then $K$-theory satisfies pro cdh descent with respect to $\Sigma$ in degrees $\le 0$, i.e., the pro abelian group $\{K_n(X,X',Y_s)\}_s$ vanishes for $n\le0$.
\end{theorem}
\begin{proof}
The Zariski presheaf of spectra $X\supseteq U\mapsto K(U,X'\times_XU,Y_s\times_XU)$ satisfies Zariski descent (since $K$-theory does), thereby yielding a descent spectral sequence \[E_2^{pq}(s)=H^p(X,\cal K_{-q,(X,X',Y_s))})\Longrightarrow K_{p+q}(X,X',Y_s)\] for each $s$. But the sheaf $\cal K_{-q,(X,X',Y_s)}$ is supported on the $0$-dimensional subspace $|Y|$ (since $X'\setminus Y'\isoto X\setminus Y$), and so this spectral sequence degenerates to edge map isomorphisms \[\bigoplus_{x\in |Y|}K_n(\Spec\roi_{X,x},X'\times_X\Spec\roi_{X,x},Y_s\times_X\Spec\roi_{X,x})\Isoto K_n(X,X',Y),\] compatibly in $s$, for each $n\in\bb Z$. This reduces us to the case that $X=\Spec\roi_{X,x}$ for each $x\in Y$.

Therefore for the rest of the proof we may suppose that $X$ is the spectrum of an excellent, normal, $2$-dimensional local ring $A$; we may also suppose that $Y=V(\frak m_A)$, since this does not change the pro abelian groups which we must show vanish. Fix $s\ge1$.

As argued at the beginning of the proof of \cite[Thm.~1.9]{Weibel2001}, $X'$ is necessarily the blow-up of $X$ along an $\frak m_A$-primary ideal $J$; replacing $J$ by $J^s$ (which does not change the blow-up) we may assume $J\subseteq\frak m_A^s$. By \cite[Prop.~1.6 \& Eg.~1.7]{Weibel2001}, there exists a reduction ideal $J'\subseteq J$ for $J$ which is generated by a regular sequence. It follows from \cite[Thm.~1.5]{Weibel2001} that the resulting canonical map of blow-ups $X'=\op{Bl}_{J}X\to\op{Bl}_{J'}X$ is finite. In conclusion, we have the following tower of abstract blow-up squares:
\[\xymatrix{
Z\times_{\op{Bl}_{J'}X}X'\ar[r]\ar[d] & X'\ar[d]\\
Z\ar[d]\ar[r] & \op{Bl}_{J'}X\ar[d]\\
V(J')\ar[r] & X
}\]
where $Z:=V(J')\times_X\op{Bl}_{J'}X$.

There is an associated fibre sequence of birelative $K$-theories \[K(X, \op{Bl}_{J'}X, V(J'))\To K(X, X', V(J'))\To K(\op{Bl}_{J'}X, X', Z),\] and similarly after replacing $J'$ by any power $J'^r$: \[K(X, \op{Bl}_{J'}X, V(J'^r))\To K(X, X', V(J'^r))\To K(\op{Bl}_{J'}X, X', Z_r).\] We now make two observations:
\begin{enumerate}
\item There exists $r\gg0$ such that the canonical map \[K_n(\op{Bl}_{J'}X, X', Z_r)\To K_n(\op{Bl}_{J'}X, X', Z)\] is zero for $n\le 0$; this follows from the previous lemma since $J'$ is generated by a regular sequence of length two and so $\op{Bl}_{J'}X$ admits a cover by two open affines (alternatively, appeal to the previous remark noting that $\dim Z=1$).
\item Thomason's calculation of the $K$-theory of a blow-up along a regular immersion means that $K(X, \op{Bl}_{J'}X, V(J'))$ is trivial \cite{Thomason1993} (see also \cite[\S1]{Cortinas2008} for some further explanation on this point).
\end{enumerate}
Comparing the previous two fibre sequences with these observations in mind, we see that the canonical map $K_n(X, X', V(J'^r))\to K_n(X, X', V(J'))$ is zero for $n\le0$. But there is $s'\gg r$ such that $\frak m_A^{s'}\subseteq J'^r$, and so $K_n(X, X', V(\frak m_A^{s'}))\to K_n(X, X', V(\frak m_A^s))$ is also zero for $n\le0$, as required to complete the proof.
\end{proof}

\section{The state of the art}
\subsection{Pro excision for Noetherian rings, and consequences}
The following theorem summarises the most important recent results concerning pro excision, in which we see in particular that pro excision is always true when dealing with Noetherian rings:

\begin{theorem}\label{pro_excision_art}
Let $f:A\to B$, $I$ be an excision situation. Then the following implications hold:

\[\xymatrix{
\txt{$K$-theory satisfies pro excision in the given situation,\\  i.e., $\{K_n(A,I^s)\}_s\to \{K_n(B,I^s)\}_s$ is an isomorphism for all $n\in\bb Z$.}\\
\ar@{=>}[u]_{\txt{\rm(Geisser--Hesselholt)}}\txt{The pro abelian group $\{\Tor_n^{\bb Z\ltimes I^s}(\bb Z,\bb Z)\}_s$ is zero for all $n\ge1$.}\\
\ar@{=>}[u]_{\txt{\rm(M.)}}\txt{The pro abelian group $\{\Tor_n^{A}(A/I^s,A/I^s)\}_s$ is zero for all $n\ge1$.}\\
\ar@{=>}[u]_{\txt{\rm(Andr\'e, Quillen)}}\txt{$A$ is Noetherian, {\rm or} $I$ is a quasiregular\footnotemark ideal of $A$}
}\]
\footnotetext{This means that $I/I^2$ is flat as an $A/I$-module and the canonical map $\bigwedge_{A/I}^nI/I^2\to\Tor^{A/I}_n(A/I,A/I)$ is an isomorphism for each $n\ge0$; if $A$ is Noetherian then it is equivalent to ask that $I$ be locally generated by a regular sequence \cite[Def.~6.10]{Quillen1970}.}
\end{theorem}
\begin{proof}[Notes on the proof]
The first implication, due to Geisser--Hesselholt \cite[Thm.~3.1]{GeisserHesselholt2011} \cite[Thm.~1.1]{GeisserHesselholt2006}, is a pro version of the criterion for excision which was mentioned in Remark~\ref{remark_SW}.

The contribution of the author \cite[Thm.~1.2]{Morrow_pro_H_unitality} follows from manipulating various Tor vanishing conditions.

If $A$ is Noetherian, then the final implication is an easy consequence of the Artin--Rees lemma which seems to have been first noticed by Andr\'e and Quillen; see, e.g., \cite[Lem.~2.1]{Morrow_pro_H_unitality} for a proof. On the other hand, if $I$ is a quasiregular of $A$, then Quillen proved that the canonical map $\Tor_n^{A}(A/I^{2s},A/I^{2s})\}_s\to \Tor_n^{A}(A/I^{s},A/I^s)\}_s$ is zero (for example, if $I$ is generated by a single non-zero-divisor $t$, then the map can be represented by $A/I^{2s}\xto{\times t^s} A/I^s$, which is clearly zero); see \cite[Prop.~8.5]{Quillen1968} and \cite[Eg.~1.4]{Morrow_pro_H_unitality}.
\end{proof}

\begin{remark}
A more direct proof that $K$-theory satisfies pro excision whenever $A$ and $B$ are Noetherian is also given in the recent work of Kerz--Strunk--Tamme \cite[Corol.~4.12]{KerzStrunkTamme2016}.
\end{remark}

From the previous theorem we already obtain an interesting case of pro cdh descent:

\begin{corollary}
Let $\Sigma$ be an abstract blow-up square of Noetherian schemes, and assume that $f:X'\to X$ is a finite morphism. Then $K$-theory satisfies pro cdh descent for $\Sigma$.
\end{corollary}
\begin{proof}
By Zariski descent for $K$-theory we may immediately reduce to the case that $X$ is affine, after which the claim follows by repeating the proof of Lemma \ref{lemma_pro_cdh_in_finite_case1}.
\end{proof}

The previous theorem may also be used to establish pro cdh descent for blow-ups along regular immersions:

\begin{theorem}\label{theorem_regular_imm}
Let $\Sigma$ be an abstract blow-up square of Noetherian schemes; assume that $Y\into X$ is a regular immersion and that $X'$ is the blow-up of $X$ along $Y$. Then $K$-theory satisfies pro cdh descent with respect to $\Sigma$.
\end{theorem}
\begin{proof}
By Zariski descent for $K$-theory we may immediately reduce to the case that $X=\Spec A$ is affine and $Y=V(J)$, where $J\subseteq A$ is the ideal defined by a regular sequence $t_1,\dots,t_c\in A$. Fix $s\ge 1$.

The key observation is that $J':=(t_1^s,\dots,t_c^s)$ is a reduction ideal for $J^s$: indeed, this means that there exists $n\ge1$ satisfying $J'J^n=J^{n+1}$, and some simple combinatorics with the exponents shows that $n=c-1$ works. Therefore the canonical map $\op{Bl}_{J^s}A\to\op{Bl}_{J'}A$ is finite by \cite[Thm.~1.5]{Weibel2001}; since the blow-up along an ideal is unchanged by replacing the ideal by any power, this yields a finite morphism $X'=\op{Bl}_{J}A\to\op{Bl}_{J'}A$. Now the previous theorem implies that, for each $n\ge0$, there exists $r\gg 0$ such that the canonical map $K_n(\op{Bl}_{J'}X, X', Z_r)\to K_n(\op{Bl}_{J'}X, X', Z)$ is zero, where $Z:=V(J')\times_A\op{Bl}_{J'}A$.

Now the same comparison of the fibre sequences (together with Thomason's calculation) already used to prove Theorem \ref{theorem_Weibel} shows that the map $K_n(X,X',V(J'^r))\to K_n(X,X',V(J'))$ is zero, which is enough to complete the proof.
\end{proof}

\subsection{Pro cdh descent}
The following was the main result of \cite{Morrow_pro_cdh_descent}:

\begin{theorem}
Let $\Sigma$ be an abstract blow-up square, and assume that $X$ is either
\begin{enumerate}
\item a Noetherian, quasi-excellent $\bb Q$-scheme of finite Krull dimension; or
\item a variety over an infinite perfect field which has strong resolution of singularities.
\end{enumerate}
Then $K$-theory satisfies pro cdh descent with respect to $\Sigma$.
\end{theorem}
\begin{proof}[Notes on the proof]
Haesmeyer's method shows that infinitesimal $K$-theory, defined as the homotopy fibre of either $K\to H\!N$ or $K\to \{\TC^m(-;p)\}_m$ depending on the situation, satisfies cdh descent;\footnote{
I am grateful to Kerz for bringing to my attention that, in case (i), the proof of this assertion in \cite{Morrow_pro_cdh_descent} contains a gap when $X$ is not of finite type over a field. Indeed, the statement of \cite[Prop.~3.6]{Morrow_pro_cdh_descent} needs to be modified by assuming that $\cal E$ commutes with filtered colimits (so that its stalks are computed by its values on local rings) and by adding one of the following additional hypotheses:
\begin{enumerate}
\item either replace ``category of Noetherian quasi-excellent $\bb Q$-schemes'' by ``category of essentially finite type $A$-schemes, where $A$ is a regular excellent $\bb Q$-algebra''; or
\item assume that $\cal E$ satisfies excision for morphisms which are isomorphisms infinitely near a closed subsheme.
\end{enumerate}
Indeed, in case (a) the given proof in op.~cit.~works, since Haesemeyer's argument works over any regular, local, excellent $\bb Q$-algebra $A$ (the key observation being that hypersurfaces in $\bb P_A^d$ are Cohen--Macaulay since $A$ is regular). So, given an arbitrary abstract blow-up square of Noetherian quasi-excellent $\bb Q$-schemes, it follows that the resulting square
\[\xymatrix{\cal E(\Spec\hat\roi_{X,x})\ar[r]\ar[d] & \cal E(X'\times_X\hat\roi_{X,x})\ar[d]\\
\cal E(Y\times_X\hat\roi_{X,x})\ar[r] & \cal E(Y'\times_X\hat\roi_{X,x})
}\]
is homotopy cartesian for each point $x\in X$ (note that the completion $\hat\roi_{X,x}$ is the quotient of a regular, complete $\bb Q$-algebra, so case (a) applies). Under assumption (b) we may immediately replace $\hat\roi_{X,x}$ by $\roi_{X,x}$ since the morphism $\Spec\hat\roi_{X,x}\to\Spec\roi_{X,x}$ is an isomorphism infinitely near the closed point of $\Spec\hat\roi_{X,x}$, thereby fixing the proof.

Finally note that both $K$-theory and Hochschild homology, hence also infinitesimal $K$-theory, satisfy hypothesis (b); so the proof of [op.~cit., Thm.~3.7] then works as stated.

Using this corrected proof, the condition ``quasi-excellent'' in (i) is no longer required.
}
this reduces the problem to showing that negative cyclic homology and topological cyclic homology have suitable pro cdh descent properties. In turn, by a series of further reductions, it is ultimately enough to check that Hochschild homology satisfies pro cdh descent with respect to any abstract blow-up square of Noetherian schemes \cite[Thm.~3.3]{Morrow_pro_cdh_descent}; the proof of this latter result is intimately linked to the Artin--Rees lemma and formal functions theorem for coherent cohomology. 
\end{proof}

However, the previous theorem has been made redundant by the following recent result of Kerz, Strunk, and Tamme \cite{KerzStrunkTamme2016}:

\begin{theorem}[Kerz--Strunk--Tamme]\label{theorem_KST}
Suppose $\Sigma$ is an abstract blow-up square of Noetherian schemes; then $K$-theory satisfies pro cdh descent with respect to $\Sigma$.
\end{theorem}

\section{Continuous $K$-theory of rigid analytic varieties}
In this section $\roi_F$ denotes a complete discrete valuation ring with field of fractions $F$. We assume that the reader is familiar with the elementary theory of rigid analytic $F$-varieties.\footnote{Everything in this section works verbatim for relative rigid spaces over a Noetherian ring complete with respect to an arbitrary ideal \cite[\S5]{BoschLutkebohmert1993}
 .} We will explain how pro cdh descent allows us to propose a definition of the ``continuous'' $K$-theory of a rigid analytic $F$-variety. This is perhaps the simplest manifestation of a fact already well-known among a few experts (c.f., \cite[Prop.~26]{KerzSaitoTamme2016}), namely that pro cdh descent is closely related to the problem of defining and studying the $K$-theory of rigid analytic varieties. Heuristically, this relation is a consequence of Raynaud's theorem \cite{Raynaud1974} that taking rigid analytic generic fibres defines an equivalence of categories
\[\xymatrix{
\txt{quasi-compact, topologically finite-type \\ formal $\roi_F$-schemes, localised by admissible blow-ups}\ar[r]^-\sim & \txt{quasi-compact, quasi-separated \\ rigid analytic $F$-varieties}
}\]
Therefore invariants of rigid analytic varieties arise as invariants of formal $\roi_F$-schemes which are suitably stable by admissible blow-ups.

Our definition of the continuous $K$-theory of a rigid analytic variety begins with the affinoid case:

\begin{definition}\label{definition_K_cts}
For each affinoid $F$-algebra $R$, its {\em continuous $K$-theory} $K^\sub{cts}(R)$ is defined\footnote{Everything in this section would probably remain valid if $\op{holim}_sK(R_0/\pi^sR_0)$ were replaced by the pro spectrum $\{K(R_0/\pi^sR_0)\}_s$, though care would be required with the descent assertions in Theorem~\ref{theorem_cts}.} to be the homotopy pushout of the diagram of spectra
\[\xymatrix{
K(R_0) \ar[r]\ar[d] & K(R)\\
\op{holim}_sK(R_0/\pi^sR_0) & 
}\] where $R_0\subset R$ is any subring of definition (i.e., a open and bounded subring of $R$) and $\pi$ is any topologically nilpotent unit of $F$.
\end{definition}

\begin{lemma}
Up to weak equivalence, the previous definition depends on neither the chosen topologically nilpotent unit $\pi$ nor the chosen subring of definition $R_0$.
\end{lemma}
\begin{proof}
The independence on $\pi$ is clear, since if $\varpi$ were another topologically nilpotent unit, then the chains of ideals $\pi^sR_0$ and $\varpi^sR_0$ ($s\ge1$) would be intertwined.


As we shall now see, the independence on the subring of definition is in fact a case of pro excision for $K$-theory. Let $R_0'\subset R$ be an alternative subring of definition, which we may assume is contained inside $R_0$ (otherwise replace $R_0'$ by $R_0'\cap R_0$, which is well-known to still be a subring of definition). We must prove that the diagram
\[\xymatrix{
K(R_0') \ar[r]\ar[d] & K(R_0)\ar[d]\\
\op{holim}_sK(R_0'/\pi^sR_0')\ar[r] & \op{holim}_sK(R_0/\pi^sR_0)
}\]
is homotopy cartesian.

Since $R_0'$ is open and $R_0$ is bounded in $R$, there is an inclusion $\pi^s R_0\subseteq R_0'$ for $s\gg0$; moreover, the chains of ideals $\{\pi^s R_0\}_{s\gg0}$ and $\{\pi^s R_0'\}_{s\ge0}$ of $R_0'$ are intertwined (since $R_0'$ is bounded and each $\pi^sR_0$ is open). Thus we may replace the bottom left term in the above diagram by $\op{holim}_{s\gg0}K(R_0'/\pi^sR_0)$, and so the homotopy cartesianness of the diagram is equivalent to the canonical map \[\op{holim}_sK(R_0',\pi^sR_0)\To \op{holim}_sK(R_0,\pi^sR_0)\] being a weak-equivalence; but this follows from the overall implication of Theorem~\ref{pro_excision_art}.
\end{proof}

\begin{remark}
Readers who wish for $K^\sub{cts}(R)$ to be unambiguously defined (i.e., not merely up to weak equivalence) should take $\op{hocolim}_{R_0}$ of the current definition, where $R_0$ varies over the filtered system of all subrings of definition of $R$.
\end{remark}

Given a qcqs rigid analytic variety $X$, it may be covered (in the admissible topology) by finitely many open affinoids and we may then attempt to define the continuous $K$-theory of $X$ by gluing the above $K^\sub{cts}$ of each affinoid; the forthcoming theorem (which crucially needs pro cdh descent) shows that the result of this process is well-defined, i.e., independent of the chosen cover; we first prove the essential lemma in the case of a two element admissible cover of an affinoid:

\begin{lemma}\label{lemma_MV}
Let $U=\op{Sp}R$ be an affinoid $F$-variety, and $\{U^1=\op{Sp}R^1,\,U^2=\op{Sp}R^2\}$ an open cover of $U$ by two rational subdomains; set $R^3:=R^1\hat \otimes_RR^2$ ($:=F\otimes_{\roi_F}$ of the $\pi$-adic completion of $R_0^1\otimes_{R_0}R_0^2$, where $R_0,R_0^1,R_0^2$ are any subrings of definition of $R,R_0^1,R_0^2$ respectively). Then the square
\[\xymatrix{
K^\sub{cts}(R)\ar[d]\ar[r] & K^\sub{cts}(R^2)\ar[d]\\
K^\sub{cts}(R^1)\ar[r] & K^\sub{cts}(R^3)
}\]
is homotopy cartesian.
\end{lemma}
\begin{proof}
Fix a subring of definition $R_0\subseteq R$. Raynaud's equivalence recalled at the beginning of the section is known to be compatible with the Zariski topology on the left and the admissible topology on the right \cite[Lem.~4.4]{BoschLutkebohmert1993}; this equivalence concretely implies that there exists a blow-up $X\to\Spec R$ with centre $\subseteq V(\pi R_0)$, and affine opens $\Spec R^1_-, \Spec R^2_-\subseteq X$ such that $R^i_0:=\hat{R^i_-}$ is a subring of definition for $R^i$ ($i=1,2,3$; here we have set $\Spec R^3_-:=\Spec R^1_-\cap \Spec R^2_-$).\footnote{To prove this, apply \cite[Lem.~4.4]{BoschLutkebohmert1993} to obtain an admissible blow-up of $\op{Spf}R_0$ which has an open cover whose rigid analytic generic fibre is the admissible cover $\{U_1,U_2\}$, and then recall that any admissible blow-up of $\op{Spf}R_0$ is (essentially by definition) the formal scheme associated to a blow-up of $\Spec R_0$ with centre $\subseteq V(\pi R_0)$.
 
For example, suppose that $f,g\in R_0$ are elements which generate an open ideal of $R_0$, so that there is a corresponding open cover of $\op{Sp}R$ by the rational subdomains $\op{Sp}R\langle\tfrac fg\rangle$ and $\op{Sp}R\langle\tfrac gf\rangle$. Then we may take $X$ to be the blow-up of $R_0$ along the ideal $(f,g)$, and $\Spec R^1_-, \Spec R^2_-$ to be the usual two affine patches of $X$, namely $R^1_-=R_0[\tfrac fg]/(\text{$g$-power torsion}$) and $R^2_-=R_0[\tfrac gf]/(\text{$f$-power torsion}$).}
 
We now apply Zariski descent and pro cdh descent to make the following conclusions:
\begin{enumerate}
\item For each $i=1,2,3$, the left and right squares in the commutative diagram
\[\xymatrix{
\op{holim}_sK(R_-^i,\pi^sR_-^i)\ar[r]\ar[d] & K(R^i_-)\ar[r]\ar[d]& K(R^i_-[\tfrac1\pi])\ar[d] \\
\op{holim}_sK(R_0^i,\pi^sR_0^i)\ar[r] & K(R^i_0)\ar[r]& K(R^i)
}\]
are homotopy cartesian: the left simply because $R^i_-/\pi^s\isoto R^i_0/\pi^s$ for each $s\ge1$, and the right by a standard.... Since the cofibre of the composition of the bottom arrows is $K^\sub{cts}(R^i)$ (by Definition \ref{definition_K_cts}), it follows that the same is true of the composition of the top arrows, i.e., the natural sequence \[\op{holim}_sK(R_-^i,\pi^sR_-^i)\To K(R_-^i[\tfrac1\pi])\To K^\sub{cts}(R)\] is a fibre sequence.
\item Next, Zariski descent implies that the squares
\[\xymatrix{
\op{holim}_sK(X,X\times_{R_0}\pi^sR_0)\ar[r]\ar[d] & \op{holim}_sK(R^1_-,\pi^sR^1_-)\ar[d] \\
\op{holim}_sK(R^2_-,\pi^sR^2_-)\ar[r] & \op{holim}_sK(R^3_-,\pi^sR^3_-)
}\quad
\xymatrix{
K(R)\ar[r]\ar[d] & K(R_-^1[\tfrac1\pi])\ar[d] \\
K(R_-^2[\tfrac1\pi])\ar[r] & K(R_-^3[\tfrac1\pi])
}
\]
are homotopy cartesian.
\item Pro cdh descent (Theorem \ref{theorem_KST}) implies that \[\holim_s K(R_0,\pi^sR_0)\quis \holim_s K(X,X\times_{R_0}\pi^sR_0);\] combined with Definition \ref{definition_K_cts} for $R$ this yields a fibre sequence \[\holim_s K(X,X\times_{R_0}\pi^sR_0)\To K(R)\To K^\sub{cts}(R)\]
\end{enumerate}
Comparing the fibre sequences of (i)\&(iii) with the homotopy cartesian squares of (ii) reveals that the desired square (\dag) is indeed homotopy cartesian.
\end{proof}

\begin{theorem}\label{theorem_cts}
There is a presheaf of spectra \[K^\sub{cts}:\text{\rm quasi-compact, quasi-separated rigid analytic $F$-varieties}\To\text{\rm Spectra}\] which satisfies the following two properties:
\begin{enumerate}
\item $K^\sub{cts}(\op{Sp} R)$ is weakly equivalent to $K^\sub{cts}(R)$, for each affinoind $F$-algebra $R$.
\item $K^\sub{cts}$ satisfies Mayer--Vietoris descent for the admissible topology, i.e., whenever $\{U_1,U_2\}$ is an admissible open cover of a qcqs rigid analytic variety $X$, then the square
\[\xymatrix{
K^\sub{cts}(X)\ar[d]\ar[r] & K^\sub{cts}(U_2)\ar[d]\\
K^\sub{cts}(U_1)\ar[r] & K^\sub{cts}(U_1\cap U_2)
}\]
is homotopy cartesian.

\end{enumerate}
Moreover, the presheaf of spectra  $K^\sub{cts}$ is unique in a suitable sense, e.g., $\infty$-categorically.
\end{theorem}
\begin{proof}
The paragraph proceeding Lemma \ref{lemma_MV} explains why $K^\sub{cts}$ is  determined by (i) and (ii); we will not attempt to make this uniqueness assertion more precise here, as it most likely requires introducing an $\infty$-category of presheaves of spectra in order to keep track of the various higher coherences between homotopies.

Various hypercohomology constructions (e.g., the \v Cech or Godement approaches of Thomason \cite[\S1]{Thomason1985}) allow us to define $K^\sub{cts}(X)$, where $X$ is an arbitrary qcqs rigid analytic variety, in terms of a homotopy limit involving $K^\sub{cts}(R)$ where $\{\op{Sp}R\}$ varies over open affinoid covers of $X$. The resulting presheaf of spectra $X\mapsto K^\sub{cts}(X)$ satisfies (ii) by a standard theorem concerning hypercohomology (e.g., \cite[Thm.~1.46]{Thomason1985}), while (i) reduces to Lemma \ref{lemma_MV}. Although many readers are likely to be convinced by this sketch, we may present the details elsewhere since there seems to be little in the literature concerning the relation between different hypercohomology constructions in the case of the admissible site of a rigid analytic variety.\footnote{
We also remark that the following more highbrow point of view seems to be possible. As mentioned at the beginning of the proof of Lemma \ref{lemma_MV}, Raynaud's equivalence is compatible with the Zariski and admissible topologies; the most natural statement of this equivalence is an identification of the topos of sheaves on a rigid analytic variety as the limit of a fibred topos (in the sense of Grothendieck \cite[\S VI]{SGA_IV_II}) of sheaves on formal schemes \cite[Thm.~4.5.12]{Abbes2010}. An extension of this result to presheaves of spectra would yield a correspondence 
\xymatrix{\txt{cocartesian presheaves of spectra on qc tft \\ formal $\roi_F$-schemes satisfying Zariski descent.} \ar@{<->}[r] & 
\txt{presheaves of spectra on rigid analytic $F$-varieties \\ satisfying admissible descent}
}

An object of the left of this correspondence is the data of a presheaf of spectra $\frak X\mapsto \cal E(\frak X)$ which satisfies not only Zariski descent but also cocartesianess for each admissible blow-up $\pi:\frak X'\to\frak X$, i.e., the resulting adjoint $\cal E_{\frak X'}\to R\pi_*\cal E_{\frak X}$ (where $\cal E_\frak X$ is the restriction of $\cal E$ to $\frak X_\sub{Zar}$, and similarly for $\frak X'$) is a local equivalence. In defining such an presheaf of spectra we may restrict to those formal schemes $\frak X=\hat {\cal X}$ which arise as the completion of finite-type $\roi_F$-schemes $\cal X$; we leave it to the interested reader to check that the cocartesinness for the presheaf ``$\cal X\mapsto$ the homotopy pushout of
\[\xymatrix@=5mm{
K(\cal X)\ar[r]\ar[d] & K(\cal X\times_{\roi_F}F)\\
\holim_sK(\cal X\times_{\roi_F}\roi_F/\pi^s\roi_F)&
}\text{''}\]
is consequence of pro cdh descent, thereby inducing $K^\sub{cts}$ on the right side of the correspondence.
}
\end{proof}

\bibliographystyle{acm}
\bibliography{../Bibliography}

\end{document}